\documentclass{article}
\usepackage{amsmath,amssymb,amsthm,xspace}

\providecommand{\bysame}{\makebox[3em]{\hrulefill}\thinspace}
\newcommand{\theoremname}{\relax}

\DeclareMathOperator{\aut}{Aut}

\newcommand{\inv}{^{-1}}

\newcommand{\abs}[2][]{\left|#2\right|_{#1}}
\newcommand{\Abs}[2][]{{\left\|#2\right\|_{#1}}}
\newcommand{\gint}[1]{\lfloor#1\rfloor}
\newcommand{\sym}[1]{#1\cup#1\inv}
\newcommand{\set}[1]{\left\{#1\right\}}
\newcommand{\setst}[2]{\set{\,#1\mid#2\,}}
\newcommand{\spaan}[1]{\left\langle#1\right\rangle}

\newcommand{\aspaan}[1]{\left[#1\right]}

\newcommand{\func}[4][\to]{#2\colon#3#1#4}
\newcommand{\lp}[1]{\Z\left[#1,#1\inv\right]}

\newcommand{\sd}{\rtimes}
\newcommand{\ten}{\otimes}

\newcommand{\isom}{\cong}
\newcommand{\N}{\mathbb{N}}
\newcommand{\W}{\N\cup\set{0}}
\newcommand{\Z}{\mathbb{Z}}

\newcommand{\R}{\mathbb{R}}
\newcommand{\C}{\mathbb{C}}
\newtheorem*{nthm}{\theoremname}

\newcommand{\IM}{I_1}
\newcommand{\IMa}{I_1'}
\renewcommand{\Im}{I_2}
\newcommand{\Ima}{I_2'}
\newtheorem{prop}{Proposition}
\newtheorem{thm}[prop]{Theorem}
\newtheorem{cor}[prop]{Corollary}
\newtheorem{lem}[prop]{Lemma}
\theoremstyle{definition}
\newtheorem{defin}{Definition}
\theoremstyle{remark}
\newtheorem*{notat}{Notation}
\newtheorem*{rk}{Remark}

\begin{document}
\title{Pairs of valuations and the geometry of soluble groups}
\author{Andrew~D. Warshall\thanks{We thank our advisor, Andrew Casson,
    for his helpful comments and Gilbert Baumslag and Tullia Dymarz
    for suggesting the problem that motivated this
    work.}\\\texttt{andrew.warshall@yale.edu}}
\maketitle

\begin{abstract}
We introduce the concepts of a pair of valuations and a good
generating set and show how they can be used to prove geometric
properties of soluble groups.
\end{abstract}

\section{Introduction}
Let $(A,a_0)$ be a pointed path-metric space. Let $\Abs{a}=d(a,a_0)$
and $B(n)$ (respectively $\overline{B}(n)$) be the open
(resp.\ closed) ball of radius $n$ about $a_0$. We say $A$ is
\emph{$n$-almost convex} if there are pairs of points $(a_i,b_i)\in
A^2$ such that $d(a_i,b_i)\le n$ but
$\lim_{i\to\infty}d_{\overline{B}(\max(\Abs{a_i},\Abs{b_i}))}(a_i,b_i)$,
where by $d_{\overline{B}(\max(\Abs{a_i},\Abs{b_i}))}$ we mean the
path metric induced from paths restricted to
$\overline{B}(\max(\Abs{a_i},\Abs{b_i}))$.  The \emph{depth} of a
point $a\in A$ is the distance from $a$ to the complement of
$B(\Abs{a})$.  We say $(A,a_0)$ has \emph{deep pockets} if it has
points of arbitrarily large depth. (This concept is clearly of
interest primarily for unbounded metric spaces.) Neither of these
concepts is a quasi-isometry invariant, but both are coarse
invariants.

If $G$ is a group and $S$ a generating set for $G$, the word metric
(denoted by $\Abs[S]{\cdot}$) gives $G$ the structure of a pointed
metric space. It is, admittedly, not a path-metric space, being
discrete, but its Cayley graph is. In practice, we ignore the edges of
the Cayley graph and simply pretend that sequences of group elements
at distance $1$ apart are paths. When this is done, it was shown by
Cannon in \cite{C} that the concept of $n$--almost convexity does not
depend on $n$, provided it is $\ge 2$. Therefore, we may simply refer
to a group's being \emph{almost convex} with respect to some
generating set.

Many groups are already known either not to be almost convex or to
have deep pockets. Thus, lamplighter groups are not almost convex with
respect to any generating set and have deep pockets with respect to
their standard generating set (see \cite{CT}). A similar but more
complicated and delicate argument (in \cite{W2}) shows the same thing
for $K\sd\spaan{t}$, where $K$ is a finite-rank abelian group and $t$
acts on $K$ by a hyperbolic automorphism which is an endomorphism of
some lattice.  Examples include lattices in Sol and the soluble
Baumslag-Solitar groups. (Non--almost convexity was already shown for
these examples in \cite{CFGT} and \cite{MS} respectively; deep pockets
were also shown for lattices in Sol in \cite{W1}.)

In this paper, we give a simpler and more general context in which to
view these results. In particular, we will be able to prove that
$\Z[1/6]\sd_{\cdot3/2}\spaan{t}$ is not almost convex with respect to
any generating set and has deep pockets with respect to some
generating set.

\section{Not almost convexity}
\begin{defin}
Let $l\in\N$, $K$ be a $\Z[t_1,t_1\inv,\dots,t_l,t_l\inv]$-module and
let $\IM$ and $\func{\Im}{K-\set{0}}{\R}$.  Then $(\IM,\Im)$ is a
\emph{pair of valuations} for $K$ if there are $C$, $b_1$, \dots,
$b_l\in\R$ such that, for all $k_1$, $k_2\in K-\set{0}$,
\begin{itemize}
\item $\IM(t_ik_1)=\IM(k_1)+b_i$,
\item $\Im(t_ik_1)=\Im(k_1)-b_i$,
\item $\IM({-k_1})=\IM(k_1)$ and similarly for $\Im$ and
\item $\IM(k_1+k_2)\le\max(\IM(k_1),\IM(k_2))+C$ and similarly for
  $\Im$.
\end{itemize}
\end{defin}

\begin{rk}
If $K$ has a strongly $t$-logarithmic $t$-generating set, then
$(I_{max},-I_{min})$ are a pair of valuations.
\end{rk}

\begin{rk}
If $K=\Z[1/6]$ and $t$ acts by multiplication by $3/2$, then the
$2$-adic and $3$-adic norms constitute a pair of valuations.
\end{rk}

\begin{notat}
By $\aspaan{t_1,\dots,t_l}$ we mean the free abelian group on $t_1$,
\dots, $t_l$.
\end{notat}

For $K$ any $\Z[t_1,t_1\inv,\dots,t_l,t_l\inv]$-module, we denote
elements of
\[
K\sd\aspaan{t_1,\dots,t_l}
\]
by ordered pairs $(m,k)$ with $m\in\Z^l$ and $k\in K$. Multiplication
is given by $(m_1,k_1)(m_2,k_2)=(m_1+m_2,t^{m_2}k_1+k_2)$, where by
$t^{m_2}$ we mean $t_1^{m_{2,1}}\dots t_l^{m_{2,l}}$ and by $m_{2,i}$
we mean the $i$th component of $m_2$.

\begin{lem}\label{triangle}
Let $K$ be a $\Z[t_1,t_1\inv,\dots,t_l,t_l\inv]$-module and
$(\IM,\Im)$ be a pair of valuations for $K$ with $C$, $b_1$, \dots,
$b_l$ as in the definition of a pair of valuations. Let $A$ be a
finite subset of $K$ and let $z\in\N$. For $a=t_1^{m_1}\dots
t_l^{m_l}\in\aspaan{t_1,\dots,t_l}$, let $B(a)=b_1m_1+\dots+b_lm_l$.
Then there is $D\in\N$ with the following property.  For $1\le i\le
n$, let $a_i\in\aspaan{t_1,\dots,t_l}$ be such that $a_i>0$ and
\[
\abs{B(a_1)},\abs{B(a_{i+1}-a_i)},\abs{B(a_n}\le z.
\]
Let $k=\sum_{i=1}^nt^{a_i}k_i$ for $k_i\in A$.  Then there is $p\in\W$
such that more than $2p$ of the $B(a_i)$ are
$\ge\max(\IM(k)-D-p,1)$.

Similarly, if we instead assume all $a_i<0$, there is $p\in\W$ such
that more than $2p$ of the $-B(a_i)$ are $\ge\max(\Im(k)-D-p,1)$.
\end{lem}

\begin{proof}
We prove the first paragraph; the proof of the second paragraph is
analogous.

We choose $v_1$, $v_2$, \dots $\in K$ inductively as follows. Let
$V_1$ be a subset of $\setst{i\in\N}{i\le n,a_i>0}$ formed by choosing
$2C$ integers $i$ with $B(a_i)$ greatest (where $C$ is as in the
definition of a pair of valuations) and let $v_1=\sum_{i\in
  V_1}t^{a_i}k_i$.  Then let $V_2$ similarly contain the $2C$ integers
with $B(a_i)$ next greatest and $v_2=\sum_{i\in V_2}t^{a_i}k_i$, and
so on.  This process must terminate since $n$ is finite; suppose $V_q$
is the last nonempty subset. (It is possible that $\abs{V_q}<2C$.)
Then $k=\sum_{i=1}^qv_i$.

Let $M\in\Z$ be such that, for all $p\in\W$, at most $2p$ of the $a_i$
are
\[
\ge M-I-C-\log_2C-2-p,
\]
where
\[
I=\max\setst{\IM(k)}{k\in A}.
\]
Then every $v_i$, by construction, is the sum of $2C$ terms $k'$ with
\begin{multline*}
\IM(k')\le I+\max\setst{B(a_j)}{j\in V_i}\\\le
I+M-I-C-\log_2C-2-C(i-1)=M-\log_2C-2-Ci,
\end{multline*}
so $\IM(v_i)<M-Ci$.

Under these conditions, I claim that, for all $i\in\set{1,\dots,q}$,
$\IM\left(\sum_{j=i}^qv_j\right)<M-C(i-1)$.  The proof is by
induction on $q-i$. For $i=q$ this is weaker than what we already
know.  For $i<q$ it follows easily from writing
\[
\IM\left(\sum_{j=i}^qv_j\right)\le\max\left(\IM(v_i),\IM\left(\sum_{j=i+1}^qv_j\right)\right)+C<M-Ci+C=M-C(i-1),
\]
where the strict inequality is by induction. The claim is proven.

It follows, setting $i=1$, that $\IM(k)<M$.  This proves the lemma,
letting $D=I+C+\log_2C+2$.
\end{proof}

\begin{prop}\label{fourthp}
Let $K$ be a $\Z[t_1,t_1\inv,\dots,t_l,t_l\inv]$-module and
$(\IM,\Im)$ be a pair of valuations for $K$ with $b_1,$ \dots, $b_l$
as in the definition of a pair of valuations. For $a=t_1^{m_1}\dots
t_l^{m_l}\in\aspaan{t_1,\dots,t_l}$, let $B(a)=b_1m_1+\dots+b_lm_l$.
Let $G=K\sd\aspaan{t_1,\dots,t_l}$.  Let $S$ be any finite generating
set for $G$.  Let $z=\max\setst{B(m)}{(m,k)\in\sym{S}}$.  Then there
is $F'\in\N$ such that, for every $g=(0,k)\in K\subset G$, either
$\IM(g)$ or $\Im(g)\le\Abs[S]{g}z/4+F'$.  (These make sense since
$g\in K$.)
\end{prop}

\begin{cor}\label{fourth}
Let $K$ be a $\Z[t_1,t_1\inv,\dots,t_l,t_l\inv]$-module and
$(\IM,\Im)$ be a pair of valuations for $K$ with $b_1,$ \dots, $b_l$
as in the definition of a pair of valuations. For $a=t_1^{m_1}\dots
t_l^{m_l}\in\aspaan{t_1,\dots,t_l}$, let $B(a)=b_1m_1+\dots+b_lm_l$.
Let $G=K\sd\aspaan{t_1,\dots,t_l}$.  Let $S$ be any finite generating
set for $G$.  Let $z=\max\setst{B(m)}{(m,k)\in\sym{S}}$.  Then there
is $F\in\N$ such that, for every $g=(a,k)\in G$ with $\abs{B(a)}\le
z$, either $\IM(k)$ or $\Im(k)\le\Abs[S]{g}z/4+F$.
\end{cor}

Note that the result is easy if the coefficient of $\Abs[S]{g}$ is
changed from $z/4$ to $z/2$ or, for that matter, to $Az$ for any
$A>1/4$.  For the sequel, however, we will need $z/4$.

\begin{proof}[Proof of Proposition~\ref{fourthp}]
Let $g=(0,k)=(m_1,k_1)(m_2,k_2)\dots(m_n,k_n)$, $(m_1,k_1)$, \dots,
$(m_n,k_n)\in\sym{S}$.  We will show either $\IM(g)$ or $\Im(g)\le
nz/4+F'$.  Let $a_i=\sum_{j=i+1}^nm_j$, so that
$k=\sum_{i=1}^nt^{a_i}k_i$; note that all $B(a_i)\in\R$ and
$a_0=a_n=0$ (since $\sum_{i=1}^nm_i=0$).  Either at most $n/2$ of the
$B(a_i)$ are positive or at most $n/2$ are negative.  Assume without
loss of generality that at most $n/2$ are positive. But
\[
\abs{B(a_{i+1}-a_i)}=\abs{B(m_{i+1})}\le z.
\]
It follows that, for each $j\in\W$, at most $n/2-2j$ of the $B(a_i)$
are $>jz$. Let $p=\gint{n/4-j+1}$, where $j$ is as in the preceding
sentence.  Then, for each $p\in\Z$, at most $2p$ of the $B(a_i)$ are
$>\max(nz/4+z-pz,0)$.

Decompose $k=v^++v^-$, where
\[
v^+=\sum\setst{t^{a_i}k_i}{a_i>0}
\]
and
\[
v^-=\sum\setst{t^{a_i}k_i}{a_i\le0}.
\]
Let
\[
I=\max\setst{\IM(s')}{s'\in\sym{S}}.
\]
Then $\IM(v^-)\le I+C(\log_2n+1)$, where $C$ is as in the definition
of a pair of valuations.

By Lemma~\ref{triangle}, there are $D\in\N$ and $p\in\W$, with $D$
independent of $g$ and $n$, such that more than $2p$ of the $B(a_i)$
are $\ge\max(\IM(v^+)-D-p,1)\ge\max(\IM(v^+)-D-pz,1)$.  It will
follow, by the first paragraph, that
\[
\IM(v^+)-D-pz<\frac{nz}{4}+z-pz,
\]
that is that $\IM(v^+)<nz/4+D+z$.

But, by the definition of a pair of valuations,
\begin{multline*}
\IM(g)=\IM(k)=\IM(v^++v^-)\le\max(\IM(v^+),\IM(v^-))+C\\\le\max\left(\frac{nz}{4}+D+z,I+C(\log_2n+1)\right)+C<\frac{nz}{4}+F',
\end{multline*}
where $C$ is again as in the definition of a pair of valuations and
$F'$ is a constant depending only, via $z$, $I$, $D$ and $C$, on $K$
and $S$ (not on $n$ or $g$), so we are done.
\end{proof}

\begin{thm}
Let $K$ be a $\Z[t_1,t_1\inv,\dots,t_l,t_l\inv]$-module which has a
pair of valuations.  Let $G=K\sd\aspaan{t_1,\dots,t_l}$ and let $S$ be
a finite generating set for $G$.  Then $G$ is not almost convex with
respect to $S$.
\end{thm}

\begin{proof}
Let $(\IM,\Im)$ be the pair of valuations with $b_1,$ \dots, $b_l$ as
in the definition of a pair of valuations. For $a=t_1^{m_1}\dots
t_l^{m_l}\in\aspaan{t_1,\dots,t_l}$, let
\[
B(a)=b_1m_1+\dots+b_lm_l.
\]
For $g=(m,k)\in G$, let $\beta(g)=k$, $\IM(g)=\IM(k)$ and
$\Im(g)=\Im(k)$.

Let $a\ne0\in K$ and let $s=(z',k_s)\in\sym{S}$ be chosen such that
$B(z')\ge B(m')$ for all $(m',k')\in\sym{S}$.  Let $z=B(z')$. Let
\[
M=\max\setst{\IM(s')}{s'\in\sym{S}}
\]
and
\[
m=\max\setst{\Im(s')}{s'\in\sym{S}}.
\]
Then, for all $g\in G$,
\[
\IM(g)\le M+z\Abs[S]{g}+C(\log_2\Abs[S]{g}+1)
\]
and
\[
\Im(g)\le m+z\Abs[S]{g}+C(\log_2\Abs[S]{g}+1),
\]
where $C$ is as in the definition of a pair of valuations. For
$n\in\W$ and $i\in\Z$, let
$g_n(i)=s^{n+i}as^{-2n}as^n=s^{i-n}as^{2n}as^{-n}$.  Then, for
$\abs{i}\le n$, $\Abs[S]{g_n(i)}\le4n-\abs{i}+2\Abs[S]{a}$.

For $n\in\W$, define $h_n^+=g_n(J)$ and $h_n^-=g_n(-J)$, where
$J\in\W$ is a constant to be chosen later.  Let $n\ge J$. By the
formula at the end of the preceding paragraph,
\[
\Abs[S]{h_n^+}\le4n-J+2\Abs[S]{a}
\]
and
\[
\Abs[S]{h_n^-}\le4n-J+2\Abs[S]{a}.
\]
Since $h_n^+(h_n^-)\inv=s^{2J}$, we have
\[
\Abs[S]{h_n^+(h_n^-)\inv}\le2J.
\]
Note that this depends only on our choice of $J$. Also,
\begin{multline*}
\IM(h_n^+)=\IM(t^{nz'}a+t^{-nz'}a+\beta(s^J))\\\le\max(\IM(t^{nz'}a),\IM(t^{-nz'}a),\IM(\beta(s^J)))+2C\\=\max(\IM(a)+zn,\IM(a)-zn,\IM(\beta(s^J)))+2C\\\le\max(\IM(a)+zn,M+Jz+C(\log_2J+1))+2C=zn+\IM(a)+2C
\end{multline*}
for $n$ large enough. But similarly
\begin{multline*}
zn+\IM(a)=\IM(t^{nz'}a)\\=\IM(h_n^+-t^{-nz'}a-\beta(s^J))\le\max(\IM(h_n^+),\IM(t^{-nz'}a,\IM(\beta(s^J)))+2C\\\le\max(\IM(h_n^+),\IM(a)-zn,M+Jz+C(\log_2J+1))+2C,
\end{multline*}
so $zn+\IM(a)-2C\le\IM(h_n^+)\le zn+\IM(a)+2C$ for $n$ large enough.
Similarly, $zn+\Im(a)-2C\le\Im(h_n^+)\le zn+\Im(a)+2C$. Also, the same
results hold by analogous arguments when $h_n^+$ is replaced by
$h_n^-$.

Every edge of the (left) Cayley graph of $G$ with respect to $S$
connects some $(m_1,k_1)$ and $(m_2,k_2)\in G$ with
$\abs{B(m_1-m_2)}\le z$.  Thus any path in the (left) Cayley graph of
$G$ with respect to $S$ connecting $h_n^-$ and $h_n^+$ must contain
some $g=(m,k)\in G$ with $\abs{B(m)}<z$.  Suppose
\[
\Abs[S]{g}\le\max(\Abs[S]{h_n^+},\Abs[S]{h_n^-})\le4n-J+2\Abs[S]{a}.
\]
Then Corollary~\ref{fourth} says that either
\[
\IM(g)\le\Abs[S]{g}z/4+F\le nz-\frac{Jz}{4}+\frac{\Abs[S]{a}z}{2}+F
\]
or the same for $\Im(g)$, where $F$ is as in that proposition.
Without loss of generality, we assume the former.

We have
\[
\IM(h_n^+)\le\max(\IM(h_n^+g\inv)+B(m),\IM(g))+C.
\]
We want to choose $J$ so that
\[
\IM(h_n^+)>\IM(g)+C.
\]
Since $\IM(g)+C\le nz-Jz/4+\Abs[S]{a}z/2+F+C$, it will suffice to take
\[
nz-\frac{Jz}{4}+\frac{\Abs[S]{a}z}{2}+F+C<nz+\IM(a)-2C,
\]
that is $J>(4/z)(F+\Abs[S]{a}z/2-\IM(a)+3C)$. Note that this is
independent of $n$.  Then we will have
$\IM(g)+C<nz+\IM(a)-2C\le\IM(h_n^+)$, as desired.  It will follow that
$\IM(h_n^+)\le\IM(h_n^+g\inv)+B(m)+C$, so
\[
nz\le\IM(h_n^+)-\IM(a)+2C\le\IM(h_n^+g\inv)+B(m)-\IM(a)+3C,
\]
that is $\IM(h_n^+g\inv)\ge nz-B(m)+\IM(a)-3C>nz-z+\IM(a)-3C$.  But
\[
\IM(h_n^+g\inv)\le
M+z\Abs[S]{h_n^+g\inv}+C\log_2(\Abs[S]{h_n^+g\inv}+1),
\]
so $\Abs[S]{h_n^+g\inv}$ goes to infinity as $n$ does, completing the
proof.
\end{proof}

\begin{cor}
Let $G$ be either
\begin{itemize}
\item a lamplighter group,
\item $\Z[1/6]\sd_{\cdot3/2}\spaan{t}$ or
\item $K\sd\spaan{t}$, where $K$ is a finite-rank abelian group and
  $t$ acts by a hyperbolic automorphism which is an endomorphism of
  some lattice.
\end{itemize}
Then $G$ is not almost convex with respect to any finite generating
set.
\end{cor}

\begin{proof}
The first two cases are easy; the third follows from Section~6 of
\cite{W2}.
\end{proof}

\begin{cor}
Let $K$ be an indecomposable $\lp{t}$-module such that the actions of
$t$ and $t\inv$ each have at least one (complex) eigenvalue of
absolute value $>1$.  Then $K\sd\spaan{t}$ is not almost convex with
respect to any generating set.
\end{cor}

\begin{proof}
Let $K_+$ and $K_-\subset K\ten\C$ be the eigenspaces corresponding to
the indicated eigenvalues $\lambda_+$ and $\lambda_-$. Let $\phi_+$
and $\phi_-$ be the projections to $K_+$ and $K_-$ that come from
sending all other eigenspaces to $0$. Then, for $k\in K$, we define
$\IM(k)=\log_{\abs{\lambda_+}}\Abs{\phi_+(k)}$ and
$\Im(k)=\log_{\abs{\lambda_-}}\Abs{\phi_-(k)}$, where $\Abs{\cdot}$
refers to any norm.  These are well-defined since $K$ is
indecomposable, so that $\phi_+(k)$ and $\phi_-(k)$ are $0$ only when
$k=0$.  It is then clear $(\IM,\Im)$ satisfies the conditions,
\end{proof}

\section{Deep pockets}
\begin{defin}
Let $K$ be a $\lp{t}$-module and let $(\IM,\Im)$ and $(\IMa,\Ima)$ be
two pairs of valuations for $K$ with the same $b=b_1$.  Then
$A\subseteq K$ is a \emph{good generating set} for $K$ with
\emph{fuzziness} $(F,F')$ if
\begin{itemize}
\item $0\in A$,
\item $A=-A$,
\item for all $a\in A-\set{0}$, $\IMa(a)\le\IM(a)$ and
  $\Ima(a)\le\Im(a)$ and
\item for all $k\in K$ with $\IMa(k)\le\IM(k)+F$ and
  $\Ima(k)\le\Im(k)+F$, there are $a_i\in A$ for all $i\in\Z$ and
  $i_j\in\Z$, $a_j\in A$ for $1\le j\le F'$ such that
  \begin{itemize}
  \item $a_i=0$ except for $\min(-\Im(k),0)-F'\le
    bi\le\max(\IM(k),0)+F'$,
  \item $k-\sum_it^ia_i=\sum_{j=1}^{F'}t^{i_j}a_j$ and
  \item $\min(-\Im(k),0)-F'\le bi_j\le\max(\IM(k),0)+F'$ for all $j$.
  \end{itemize}
\end{itemize}
\end{defin}

\begin{rk}
If $K$ has a strongly $t$-logarithmic $t$ generating set, then for any
$F$ it is also a good generating set with fuzziness $(F,F')$ for some
$F'$; just let $\IM=\IMa=I_{max}$ and $\Im=\Ima=-I_{min}$.
\end{rk}

\begin{thm}
For every $b$ and $C$ there is $F$ with the following property.  Let
$K$ be a $\lp{t}$-module and let $(\IM,\Im)$ and $(\IMa,\Ima)$ be two
pairs of valuations for $K$ with $C$ as given and the same $b=b_1$.
Let $A\subseteq K$ be a finite good generating set for $K$ with
fuzziness $(F,F')$ for some $F'$. Then $K\sd\spaan{t}$ has deep
pockets with respect to $\setst{ata'}{a,a'\in A}\cup A$.
\end{thm}

\begin{proof}
Let $S=\setst{ata'}{a,a'\in A}\cup A$. Let
\[
M=\max\setst{\abs{\IM(a')},\abs{\Im(a')},\abs{\IMa(a)},\abs{\Ima(a)}}{a'\in
  A}.
\]

Let $a\in A$. Let $k_i=t^ia+t^{-i}a\in K\subset K\sd\spaan{t}$. Then
$\IM(k_i)$, $\Im(k_i)\ge i-C-M$, so, by Proposition~\ref{fourthp},
$\Abs[S]{k_i}\ge4(i-C-M-H)/\abs{b}$, where $H$ is the quantity
referred to as $F'$ in the statement of that proposition. (In
particular, $H$ is independent of $i$.) Let $l\in G$ be such that
$\Abs[S]{lk_i\inv}=j_i$.  Then $lk_i\inv=t^nk'$ for some $n\in\Z$ and
$k'\in K$.  Furthermore, $\abs{n}\le j_i$ and $\IM(k')$, $\IMa(k')$,
$\Im(k')$ and $\Ima(k')\le\abs{b}j_i+C(\log_2j_i+1)+M$.  It follows
that $l=t^nk''$ with
\begin{itemize}
\item $\abs{n}\le j_i$,
\item $\IM(k'')$, $\Im(k'')\le i+M+2C$ and
\item $\IMa(k'')$, $\Ima(k'')\ge i-M-2C$
\end{itemize}
so long as $i>\abs{b}j_i+C(\log_2j_i+3)+2M$.

Then the definition of a good generating set (with $F=2M+4C$) gives
$a_i\in A$ for all $i\in\Z$ and $i_j\in\Z$, $a_j\in A$ for $1\le i\le
F'$ such that
\begin{itemize}
\item $a_i=0$ except for $\min(-i-M-2C,0)-F'\le
  bi\le\max(i+M+2C,0)+F'$,
\item $k''-\sum_it^ia_i=\sum_{j=1}^{F'}t^{i_j}a_j$ and
\item $\min(-i-M-2C,0)-F'\le bi_j\le\max(i+M+2C,0)+F'$ for all $j$.
\end{itemize}
Thus $\Abs[S]{k''}\le4i+4M+8C+5F'$.  We are done by the Fuzz Lemma
from \cite{W1}, since the upper bound on $j_i$ goes to infinity as $i$
does.
\end{proof}

\begin{cor}
Let $G$ be a lamplighter group or $K\sd\spaan{t}$, where $K$ is a
finite-rank abelian group and $t$ acts by a hyperbolic automorphism
which is an endomorphism of some lattice. Then $G$ has deep pockets
with respect to some generating set.
\end{cor}

\begin{proof}
The first case is easy; the second again requires Section~6 of
\cite{W2}.
\end{proof}

\begin{cor}
Let $K\supset L\isom\Z^n$, $P$, $Q\in\aut(K)$ such that
\begin{itemize}
\item $\abs{\det(P)}$ and $\abs{\det(Q)}$ are coprime,
\item $PQL=QPL$,
\item $K=\bigcup_{i=-\infty}^0P^iQ^iL$,
\item $0=\bigcap_{i=0}^\infty P^iQ^iL$ and
\item all eigenvalues of $QP\inv$ have absolute value $>1$.
\end{itemize}
Let $t$ act by $QP\inv$. Then $K\sd\spaan{t}$ has deep pockets with
respect to some generating set.

In particular, $\Z[1/6]\sd_{\cdot3/2}\spaan{t}$ has deep pockets with
respect to some generating set.
\end{cor}

\begin{proof}
Let $B$ be the symmetrized closed unit cube in $K$ with respect to an
$n$-element generating set for $QL$. (In what follows, we will use
$\Abs{\cdot}$ to denote $L_1$ norm with respect to this generating
set.)  Let $\IMa(k)=\min\setst{i}{k\in(QP\inv)^iB}$. If
$\abs{\det(P)}=1$ then let $\IM(k)=\IMa(k)$. Let
\[
\IM(k)=-\max\setst{i}{k\in\bigcup_{j=-\infty}^0P^iQ^jL}
\]
if $\abs{\det(P)}>1$ and
\[
\Ima(k)=\Im(k)=\max\setst{i}{k\in\bigcup_{j=-\infty}^0P^jQ^iL}.
\]
(We have $\abs{\det(Q)}>1$ by the last condition on $P$ and $Q$.)
Clearly, $(\IM,\Im)$ and $(\IMa,\Ima)$ are pairs of valuations with
$b=1$ and $C=\gint{\log_\lambda2}+1$, where $\lambda$ is the least
absolute value of any eigenvalue of $QP\inv$.

Let $A=L\cap B$. I claim that $A$ is a fuzzy good generating set for
$K$.  The first three conditions are clearly satisfied.  Let $k\in K$
be such that $\IMa(k)\le\IM(k)+F$. Let
$N=Q^{\IM(k)-\Im(k)}(PQ\inv)^{\IM(k)}$.  Then $N(k)\in L$ (it is here
that we use that the determinants are coprime) and
\[
N(k)\in
Q^{\IM(k)-\Im(k)}(PQ\inv)^{\IM(k)}(QP\inv)^{\IM(k)+F}B=Q^{\IM(k)-\Im(k)}F''B,
\]
where $F''=(QP\inv)^F$ depends only on $F$.

It remains to show there are $a_i$ and $a_j\in A$ and $i_j\in\Z$ such
that 
\begin{itemize}
\item $a_i=0$ except for $\min(-\Im(k),0)-F'\le
  i\le\max(\IM(k),0)+F'$,
\item $\min(-2\Im(k),0)-F'\le i_j\le\max(\IM(k)-\Im(k),0)+F'$ for all
  $j$ and
\item\begin{multline*}
  N\left(k-\sum_i(QP\inv)^{i+\Im(k)}a_i\right)\\=N(k)-\sum_iQ^{\IM(k)-\Im(k)}(QP\inv)^{i+\Im(k)-\IM(k)}(a_i)\\=N\left(\sum_{j=1}^{F'}(QP\inv)^{i_j+\Im(k)}a_j\right)\\=\sum_{j=1}^{F'}Q^{\IM(k)-\Im(k)}(QP\inv)^{i_j+\Im(k)-\IM(k)}(a_j).
\end{multline*}
\end{itemize}
But clearly there are $a_i\in A$ such that
\[
k'=N(k)-\sum_{i=0}^{\IM(k)-\Im(k)}Q^{\IM(k)-\Im(k)}(QP\inv)^{i+\Im(k)-\IM(k)}(a_i)\in Q^{\IM(k)-\Im(k)}L,
\]
and $k'$ is the difference of two terms $\in F'''Q^{\IM(k)-\Im(k)}B$,
where $F'''\in\N$ depends only on $F$ (via $F''$). Thus for any $F$
there is $F'$ such that $A$ is a finite good generating set with
fuzziness $(F,F')$, so we are done.

The last sentence follows trivially.
\end{proof}


\begin{thebibliography}{9}
\bibitem{C} J.W.~Cannon. Almost convex groups. \emph{Geometriae
Dedicata}, 22(2):197--210, 1987.
\bibitem{CFGT} J.W.~Cannon et al. Solvgroups are not almost convex.
\emph{Geometriae Dedicata}, 31(3):291--300, 1989.
\bibitem{CT} S.~Cleary and J.~Taback. Dead end words in lamplighter
groups and other wreath products. \emph{The Quarterly Journal of
Mathematics}, 56(2):165--78, 2005. \texttt{arXiv:math.GR/0309344}.
\bibitem{MS} C.F.~Miller and M.~Shapiro. Solvable Baumslag-Solitar
groups are not almost convex. \emph{Geometriae Dedicata},
72(2):123--7, 1998. \texttt{arXiv:math.GR/9702202}.
\bibitem{W1} A.D.~Warshall. Deep pockets in lattices and other groups.
  To appear in \emph{Transactions of the American Mathematical
    Society}.  \texttt{arXiv:math.GR/0611575}, 2007.
\bibitem{W2} \bysame. Strongly $t$-logarithmic $t$-generating sets:
  Geometric properties of some soluble
  groups. Preprint. \texttt{arXiv:0808.2789}, 2008.
\end{thebibliography}
\end{document}